\documentclass[a4paper,12pt,reqno]{amsart}
\usepackage{amsmath,amsfonts,amssymb,amsthm,enumerate,multicol}
\usepackage{tikz,graphicx}
\usepackage{hyperref}
\usepackage{caption,enumitem,wasysym}
\usepackage{cleveref}
\usepackage{epstopdf}
\usepackage{float,wrapfig}
\usepackage[labelsep=space]{caption}
\usepackage[font=footnotesize]{caption}
\usepackage{xcolor}

\newtheorem{theorem}{Theorem}[section]

\newtheorem{lemma}{Lemma}[section]

\newtheorem{remark}{Remark}[section]

\allowdisplaybreaks

 \theoremstyle{plain}
\newtheorem{thm}{Theorem}[section]

\newtheorem{prop}[thm]{Proposition}

\DeclareMathOperator{\sgn}{sgn}

\usepackage{mathtools}

\usepackage{amsmath,amsfonts,amssymb,amsthm,enumerate,multicol}
\usepackage{tikz}
\usepackage{float}
\usepackage{autobreak}

\allowdisplaybreaks

\numberwithin{equation}{section}

\begin{document}
\begin{center}
\Large{\textbf{ On Classical solution to the Discrete Coagulation Equations with Collisional Breakage }}
\end{center}





\medskip
\medskip
\centerline{${\text{Mashkoor~ Ali}}$ and  ${\text{Ankik ~ Kumar ~Giri$^*$}}$ }\let\thefootnote\relax\footnotetext{$^*$Corresponding author. Tel +91-1332-284818 (O);  Fax: +91-1332-273560  \newline{\it{${}$ \hspace{.3cm} Email address: }}ankik.giri@ma.iitr.ac.in}
\medskip
{\footnotesize

  \centerline{ ${}^{}$Department of Mathematics, Indian Institute of Technology Roorkee,}
   \centerline{Roorkee-247667, Uttarakhand, India}

}

\bigskip

\begin{quote}
{\small {\em \bf Abstract.} In this article, the existence of global classical solutions to the discrete coagulation equations with collisional breakage is established for collisional kernel having linear growth whereas the uniqueness is shown  under additional restrictions on collisional kernel. Moreover, mass conservation property and propagation of moments of  solutions are also discussed.}
\end{quote}


\vspace{.3cm}






\section{\textbf{INTRODUCTION}}
Coagulation and breakage  models aim at describing the mechanisms by which clusters combine to form bigger clusters or break into smaller fragments. These models are used to explain a wide range of phenomena, such as cloud droplets formation \cite{LG 1976,SRC 1978} and planet formation \cite{BRILL 2015,SV 1972}.   
Each cluster in these situations is fully characterized by a size variable (volume, mass, number of clusters, etc.) that can be either a positive real number (continuous models) or a positive integer (discrete models). The clusters we are looking at are discrete in the sense that  they are made up of a finite number of fundamental building blocks (monomers) having unit mass.
  In nature, when we examine a very short period of time, coagulation is binary, whereas breakage can occur in two ways: linear (spontaneous) or non-linear. The linear breakage process is governed solely by clusters properties (and also by external forces, if any), whereas the non-linear breakage process occurs when two or more clusters encounter and matter is transferred between them. As a result, the mass of the emerging cluster in a non-linear breakage process may be larger than the colliding clusters.
The discrete coagulation equations with collisional breakage describe the time evolution of concentration $w_i(t)$ for the  $i^{th}$-cluster at time $t\geq 0$ and read

\begin{align}
\frac{dw_i}{dt}  =& \frac{1}{2}\sum_{j=1}^{i-1} p_{j,i-j} a_{j,i-j} w_j w_{i-j} -\sum_{j=1}^{\infty} a_{i,j} w_i w_j \nonumber \\
&+ \frac{1}{2} \sum_{j=i+1}^{\infty} \sum_{k=1}^{j-1} N_{j-k,k}^i(1- p_{j-k,k}) a_{j-k,k} w_{j-k} w_k, \hspace{.5cm} i \in \mathbb{N}, \label{DCBEC}\\
w_i(0) &= w_i^0, \hspace{.5cm} i \in \mathbb{N}.\label{IC}
\end{align}

The first term on the right-hand side of  \eqref{DCBEC} accounts for the appearance of $i$-clusters through collision and coagulation of smaller ones, while the second term accounts for their disappearance due to collisions with other clusters. The third term describes the  appearance  of $i$-clusters after the collision and breakup of larger clusters. Here $a_{i, j}$ denotes the rate by which  clusters of size  $i$  collide with clusters  of size $j$ and $p_{i,j}$ is the probability of the event that the two colliding clusters of sizes $i$ and $j$ join to form a single cluster. If this does not occur, clusters fragment with the possibility of matter transfer, and this event occurs with the probability $(1-p_{ i,j})$. The distribution function of the generated fragments, $\{N_{i,j}^s, s=1,2,...,i+j-1\}$, has the properties listed below.

\begin{align}
N_{i,j}^s = N_{j,i}^s \geq 0 \hspace{.7cm} \text{and} \hspace{.7cm} \sum_{s=1}^{i+j-1} s N_{i,j}^s = i+j. \label{LMC}
\end{align}
The second term in \eqref{LMC} infer that   mass is conserved during each collisional breakage event. 

For a solution $ w(t) = (w_i(t))_{i \in \mathbb{N}}$ of \eqref{DCBEC}, we define the $r$-th moment as
\begin{align} \label{Moments}
\mathcal{M}_r(w(t))=\mathcal{M}_r(t) :=\sum_{i=1}^{\infty} i^r w_i(t)  \hspace{.3cm} \text{for} \hspace{.3cm} r \geq 0.
\end{align}
In the above equation \eqref{Moments}, the zeroth $(r=0)$ and first $(r=1)$ moments denotes the total number of particles and the total mass of particles, respectively, in the system. Before going any further, it is important to note that in the absence of fragmentation $(p_{i, j} = 1)$, the system \eqref{DCBEC}--\eqref{IC} is Smoluchowski coagulation equation  which has been widely studied by physicists and mathematicians (see, e.g., \cite{COSTA 2015} and the references therein). Since, it is known that the total mass of particles are neither created nor destroyed in the reactions described by \eqref{DCBEC}. Thus it is expected that total the total mass $\mathcal{M}_1(t)$ remain conserved throughout the time evolution.
In the last few decades, the linear(spontaneous) fragmentation equation with coagulation has received a lot of attention which intially studied by Filippov \cite{FAF 61}, Kapur \cite{KAPUR 72}, McGrady and Ziff \cite{MCG 87, ZRM 85}. In \cite{BAN 12, BAN 19, BAN 19I, BAN 12I, LAMB 20}, semigroup technique has been employed to study the existence and uniqueness of classical  solutions to linear fragmentation equation with coagulation having  appropriate assumptions on coagulation and fragmentation kernels whereas in \cite{CARR 92, CARR 94, Laurencot 1999, Laurencot 2001, Laurencot 2002}, issues related to existence and uniqueness of weak solutions to coagulation equation with spontaneous fragmentation have been investigated by using weak $L^1$ compactness method (for more information, see \cite{BLL 2019} and  references therein). The nonlinear breakage equation, on the other hand, has not been studied to that level. In \cite{CHNG 90}, Cheng and Redner discussed  the  dynamics of continuous, linear and collision-induced nonlinear fragmentation events. For a linear fragmentation process, they looked at scaling theory to characterise the evolution of the cluster size distribution, whereas for a non linear fragmentation process, they have examined the asymptotic behaviour of a class of models in which a two-particle collision causes both particles to break into two equal parts, just the larger particle to split in two, or only the smaller particle to split. Later, Krapivsky and Ben-Naim \cite{Krapivsky 2003} studied the kinetics of nonlinear collision-induced fragmentation, obtaining the fragment mass distribution analytically using the travelling wave behaviour of the nonlinear collsion equation. They have also demonstrated that the system goes through a shattering transition, in which a finite part of the mass is lost to fragments of infinitesimal sizes. Lauren\c{c}ot and Wrzosek \cite{Laurencot 2001} discussed the existence, uniqueness, mass conservation, and the large time behaviour of weak solutions to \eqref{DCBEC}--\eqref{IC} with suitable restrictions on the collision kernel and probability function using the weak $L^1$ compactness technique. Using the same technique, one of the authors has extensively studied the continuous version of the \eqref{DCBEC}--\eqref{IC} in recent years in \cite{PKB 2020I, PKB 2021, PKB 2021I, AKG 2021, AKG 2021I}, where the existence and uniqueness of weak solutions have been discussed for different classes of collision kernels, while global classical solution is discussed in \cite{PKB 2020}. The present work is motivated from \cite{Laurencot 2001} and the goal of  this paper  is to prove the existence, uniqueness  and propagation of moments  of the classical solution to  equation \eqref{DCBEC}--\eqref{IC} using the approach developed in \cite{PBD 77}.

The paper is organized as follows. Section 2 contains definition, assumptions on kernels and initial data and the statement of the  main theorem, while Section 3 covers the local existence theorem, moment estimates, time equicontinuity, and the proof of  main theorem. In Section 4, it is  shown that the solution is unique, and finally the propogation of moments is explored in  Section 5.

\section{\textbf{ Function spaces and Assumptions }}

Fix  $T \in (0,\infty)$.  Let $\Omega_{\beta}(T)$ be the space of  $ w = (w_i)_{i \in \mathbb{N} } \in \mathcal{C}([0,T],l_{1}^{\beta})$ with bounded norm defined by
$$ \| w \| := \sup_{t \in [0,T]} \sum_{i=1}^{\infty} i^{\beta}|w_i(t)|, \hspace{.4cm}\text{where} \hspace{.4cm} 1 < \beta \leq 2,$$
and its positive cone
$$ \Omega_{\beta}^+(T) = \{ w =(w_i)_{i\in \mathbb{N}} \in \Omega_{\beta}(T) : w_i \geq 0 \hspace{.2cm}  \text{for each } \hspace{.2cm} i \geq 1\}. $$

In order to show the existence, mass conservation property and propagation of moments of  classical solution to (\ref{DCBEC})--(\ref{IC}), assume that the collision kernel $a_{i,j}$ and the probability function $p_{i,j}$ are non-negative and  symmetric i.e.\
\begin{align}\label{ASSUM}
0\leq a_{i,j}= a_{j,i} \leq A (i +j ) \hspace{.2cm} \text{and}  \hspace{.2cm} 0\leq p_{i,j}= p_{j,i} \leq 1 \hspace{.2cm}  \forall i,j \geq 1,
\end{align}
where $A$ is a positive constant. 
Moreover, for proving the uniqueness of classical solution to \eqref{DCBEC}--\eqref{IC}, the following additional restriction on the collision kernel is required
\begin{align}\label{ASSUM1}
0\leq a_{i,j}= a_{j,i} \leq B(i^{\gamma} +j^{\gamma}), \hspace{.2cm}\text{where} \hspace{.2cm}  \gamma \in [0,1] \hspace{.2cm} \text{and}\hspace{.2cm} 1< 1+\gamma \leq \beta,
\end{align}
where $B$ is a positive constant.

We are now ready to state the main theorem:
\begin{theorem}(Main Theorem)\label{MT}
Assume that the collision rate $a_{i,j}$ and the probability function $p_{i,j}$ satisfy \eqref{ASSUM} and initial data $w^0=(w_i^0)_{i\in \mathbb{N}} \in \Omega_{\beta}^+(0)$. Further assume that there is a constant $\alpha_0$ such that 
\begin{align}\label{FNP}
N_{i,j}^s  \leq \alpha_0 \hspace{.5cm} \text{for} \hspace{.5cm} s=1,2,...i+j-1.
\end{align}
Then there exists a  mass conserving classical solution to \eqref{DCBEC}--\eqref{IC} in $\Omega_{\beta}^+(T)$. In addition, if \eqref{ASSUM1} holds, then the classical solution to \eqref{DCBEC}--\eqref{IC} is unique.
\end{theorem}

\begin{remark}
A typical example satisfying \eqref{FNP} is given by 
\begin{align*}
N_{i,j}^s := (\nu +2) \frac{s^{\nu}}{(i+j)^{\nu + 1}}, \hspace{.2cm} (i,j) \in \mathbb{N}^2,\hspace{.2cm}  s=1,2,...i+j-1 \hspace{.2cm}  \text{and} -1 < \nu \leq 0.
\end{align*}

Clearly \eqref{FNP} holds for $\alpha_0 = 2$.
\end{remark}

Now, in the next section, the existence of solutions to \eqref{DCBEC}--\eqref{IC} is shown by passing to the limit of solutions to the truncated finite dimensional system of \eqref{DCBEC}--\eqref{IC}. Most of the technical tools which  have been utilised to prove the main theorem were adapted from \cite{PBD 77} and \cite{SIW 1989}, with appropriate modifications.

\section{\textbf{Existence of Classical Solution}}

 In this section we established the existence of solution for the truncated system of \eqref{DCBEC}--\eqref{IC}. For $ n \geq 2$, consider the following truncated system of $n$  ordinary differential equations,
\begin{align}
\frac{d}{dt}w_i^n =& \frac{1}{2}\sum_{j=1}^{i-1} p_{j,i-j} a_{j,i-j} w_j^n w_{i-j}^n -\sum_{j=1}^{n-i} a_{i,j} w_i^n w_j^n \nonumber \\
&+ \frac{1}{2} \sum_{j=i+1}^{n} \sum_{k=1}^{j-1} N_{j-k,k}^i(1- p_{j-k,k}) a_{j-k,k} w_{j-k}^n w_k^n, \label{Trunc} \\
\text{with truncated initial data} \nonumber\\
 w_i^n(0) &= w_i^0, \label{TruncIC}
\end{align}
 for $ i \in \{1,...,n\}.$
 
Proceeding as in \cite{BALL 90}, Lemmas 2.1 and 2.2 we obtain the following result.

 \begin{theorem}\label{LET}
 Let $0<T <+\infty $ be given and the collision kernel satisfies assumption \ref{ASSUM}. Then for each $n \geq 2$  the system \eqref{Trunc}--\eqref{TruncIC} has a unique solution $w^n=\{w_i^n\}_{i=1}^{n} \in \mathcal{C}^1([0,T],\mathbb{R}^n) $.
 Furthermore, the truncated mass conservation law is also satisfied, i.e.
\begin{align}\label{TMC}
\sum_{i=1}^{n} i w_i^n(t) = \sum_{i=1}^{n} i w_i^n(0) \hspace{.3cm} \text{for all} \hspace{.3cm} t \in [0,T].
\end{align}
\end{theorem}

From Theorem \ref{LET}, for each $n\geq 2$,  we obtain that, $w^n$ is unique non negative classical solution of \eqref{Trunc}--\eqref{TruncIC}. Next we define the zero extension of each solution $w^n$, i.e.
 \begin{align}
 \tilde{w}^n_i(t) =
\begin{cases}
  w_i^n(t),  &1\leq i\leq n \\ 
    0,         &i>n.     
\end{cases}
\end{align}
 
For the rest of the paper, we will drop the $\tilde{.}$ notation. With this convention $w^n=\{w_i^n\}_{i\in \mathbb{N}}$ satisfies the equation 
\begin{align}
\frac{dw_i^n}{dt}  =& \frac{1}{2}\sum_{j=1}^{i-1} p_{j,i-j} a_{j,i-j} w_j^n w_{i-j}^n -\sum_{j=1}^{\infty} a_{i,j} w_i^n w_j^n \nonumber \\
&+ \frac{1}{2} \sum_{j=i+1}^{\infty} \sum_{k=1}^{j-1} N_{j-k,k}^i(1- p_{j-k,k}) a_{j-k,k} w_{j-k}^n w_k^n, \label{DCBECTIL}\\
w_i^n(0) &= w_i^0 ,\label{ICTIL}
\end{align}

It follows from Theorem \ref{LET} that there exists a non-negative unique classical solution $w^n(t)=\{w_i^n(t)\}_{i\in \mathbb{N}}$ to \eqref{DCBECTIL}--\eqref{ICTIL} in  $\Omega_{\beta}^+(T)$ which satisfies the mass conservation law. And for this solution we define the $r$-th moment as
\begin{align*}
\mathcal{M}_{r}(w^n(t))=\mathcal{M}_{r}^n(t) := \sum_{i=1}^{\infty}i^r w_i^n(t).
\end{align*}

Now, we prove a number of results that are essential for the proof of the main theorem.

\subsection{\textbf{Uniform boundedness of approximated moments}}
Fix  $T \in (0,\infty)$ and $w^0 \in \Omega_{\beta}^+(0)$. 
Multiplying equation \eqref{DCBECTIL} with the weight $g_i$ and taking the summation over $i$, we get
\begin{align}
\sum_{i=1}^{\infty} g_i \frac{dw_i^n }{dt} = &\frac{1}{2}\sum_{i=1}^{\infty} \sum_{j=1}^{\infty}  ( g_{i+j} - g_i - g_j) a_{i,j} w_i^n w_j^n  \nonumber\\
& -\frac{1}{2} \sum_{i=1}^{\infty} \sum_{j=1}^{\infty} (1-p_{i,j}) \Big(g_{i+j} - \sum_{s=1}^{i+j-1} g_s N_{i,j}^s  \Big) a_{i,j} w_i^n w_j^n. \label{GME}
\end{align}

Putting $g_i = i $ in \eqref{GME} and using \eqref{LMC} gives
\begin{align*}
\frac{d \mathcal{M}_{1}^n(t)}{dt}=0.
\end{align*}
After integrating with respect to $t$, we obtain 
\begin{align}
\mathcal{M}_{1}^n(t) = \sum_{i=1}^{\infty}i w_i^n(t) =  \sum_{i=1}^{\infty}i w_i^n(0) \leq \sum_{i=1}^{\infty}i w_i^0  = \Lambda_1. \label{FM}
\end{align}
Next, taking $g_i = i^{\beta} $ in \eqref{GME}, we find that
\begin{align*}
\sum_{i=1}^{\infty} i^{\beta}\frac{dw_i^n }{dt} =& \frac{d}{dt}\sum_{i=1}^{\infty} i^{\beta} w_i^n =\frac{1}{2}\sum_{i=1}^{\infty} \sum_{j=1}^{\infty}  ( (i+j)^{\beta} - i^{\beta} - j^{\beta}) a_{i,j} w_i^n w_j^n  \nonumber\\
& -\frac{1}{2} \sum_{i=1}^{\infty} \sum_{j=1}^{\infty} (1-p_{i,j}) \Big((i+j)^{\beta} - \sum_{s=1}^{i+j-1} s^{\beta} N_{i,j}^s  \Big) a_{i,j} w_i^n w_j^n.  \nonumber 
\end{align*}
Thanks to the inequality 
\begin{align*}
 \sum_{s=1}^{i+ j-1} s^{\beta} N_{i,j}^s \leq  (i+j)^{\beta},
 \end{align*}
 which is applied in the last term on the right-hand side of above equation and then from \eqref{ASSUM}, we infer that 
\begin{align}
\frac{d}{dt}\sum_{i=1}^{\infty} i^{\beta} w_i^n \leq \frac{A}{2}\sum_{i=1}^{\infty} \sum_{j=1}^{\infty}  ( (i+j)^{\beta} - i^{\beta} - j^{\beta}) (i+j) w_i^n w_j^n.  \label{BEFINEQ}
\end{align}
Further, let us utilize the following inequality  from \cite[Lemma 2.3]{CARR 92}, there is a constant $\jmath_{\beta}>0$ that depends only on $\beta$, such that 
 \begin{align*}
 (i+j)((i+j)^{\beta} - i^{\beta} - j^{\beta}) \leq \jmath_{\beta} ( ij^{\beta} + i^{\beta} j), 
 \end{align*}
in \eqref{BEFINEQ} to obtain
\begin{align*}
\frac{d}{dt}\sum_{i=1}^{\infty} i^{\beta} w_i^n \leq& \frac{A\jmath_{\beta}}{2}\sum_{i=1}^{\infty} \sum_{j=1}^{\infty}   ( ij^{\beta} + i^{\beta} j) w_i^n w_j^n,\\
\end{align*}
which can be rewritten as
\begin{align}
\frac{d}{dt} \mathcal{M}_{\beta}^n(t) &\leq A\jmath_{\beta} \mathcal{M}_{\beta}^n(t) \mathcal{M}_1^n(t).\label{betammnteq}
\end{align}

On integrating the above inequality and using \eqref{FM}, we obtain
\begin{align}\label{betammnt}
\mathcal{M}_{\beta}^n(t) \leq \Lambda_{\beta}(T),
\end{align}
where $ \Lambda_{\beta}(T) := \mathcal{M}_{\beta}(0) \exp (A \jmath_{\beta}  \Lambda_1 T)$.\\

Now we will show that $w_i^n(t)$ is uniformly bounded for each $i \in \mathbb{N}$ and $ t \in [0,T]$, where $T<\infty $  is arbitrary.

\subsection{\textbf{Uniform boundedness of approximated solutions}}
 
 Due to non negativity of second term in equation \eqref{DCBECTIL}, we get
 \begin{align*}
\frac{d}{dt}w_i^n  &\leq  \frac{1}{2}\sum_{j=1}^{i-1} p_{j,i-j} a_{j,i-j} w_j^n w_{i-j}^n 
+ \frac{1}{2} \sum_{j=i+1}^{\infty} \sum_{k=1}^{j-1} N_{j-k,k}^i(1- p_{j-k,k}) a_{j-k,k} w_{j-k}^n w_k^n\\
&\leq  \frac{1}{2}\sum_{j=1}^{i-1} p_{j,i-j} a_{j,i-j} w_j^n w_{i-j}^n 
+ \frac{1}{2} \sum_{j=2}^{\infty} \sum_{k=1}^{j-1} N_{j-k,k}^i(1- p_{j-k,k}) a_{j-k,k} w_{j-k}^n w_k^n.
\end{align*}
Next, we infer from Fubini's theorem that 
\begin{align*}
\frac{d}{dt}w_i^n 
&\leq \frac{1}{2}\sum_{j=1}^{i-1} p_{j,i-j} a_{j,i-j} w_j^n w_{i-j}^n 
+ \frac{1}{2} \sum_{k=1}^{\infty} \sum_{j=1}^{\infty} N_{j,k}^i a_{j,k} w_{j}^n w_k^n.
\end{align*}

From \eqref{ASSUM}, \eqref{FNP} and \eqref{FM} it follows  that
\begin{align}
\frac{d}{dt}w_i^n  &\leq  \frac{A}{2}\sum_{j=1}^{i-1} i w_j^n w_{i-j}^n +  A \alpha_0 \Lambda_1^2 .\label{WENQ}
\end{align}

Let the upper equation of \eqref{WENQ} be
\begin{align}
\frac{d}{dt}x_i  &= \frac{A}{2}\sum_{j=1}^{i-1}  i x_j x_{i-j} +  A \alpha_0 \Lambda_1^2, \label{XEQ}\\ 
x_i(0) &= x_i^0, \label{XEQIC}
\end{align}
for $ i \in \mathbb{N} $.
\begin{lemma}\label{WLEQX}
Let $w^n(t) = (w_i^n(t))_{i\in \mathbb{N}}$ and  $x(t) = (x_i(t))_{i\in \mathbb{N}}$ be unique solution to \eqref{DCBECTIL}--\eqref{ICTIL} and \eqref{XEQ}--\eqref{XEQIC} respectively. Then  for each $i\in \mathbb{N}$, $w_i^n(t) \leq x_i(t)$, provided $ w_i^n(0)\leq x_i(0).$
\end{lemma}
\begin{proof}
For $i=1$, from \eqref{WENQ}, we obtain
\begin{align*}
\frac{d}{dt}w_1^n(t) &\leq  A \alpha_0 \Lambda_1^2,\\
w_1^n(t) & \leq w_1^n(0) +  A \alpha_0 \Lambda_1^2 t \leq x_i(0) +  A\alpha_0 \Lambda_1^2 t = x_1 (t).
\end{align*}
For $ i=2 $, we have 
\begin{align*}
\frac{d}{dt}w_2^n(t) & \leq \frac{1}{2} \sum_{j=1}^{1}  a_{j,2-j} w_j^n w_{2-j}^n +   A\alpha_0 \Lambda_1^2 \leq \frac{A}{2} w_1^n(t) w_1^n(t)+   A \alpha_0 \Lambda_1^2 \leq \frac{A}{2} ( x_1(t)^2 +2\alpha_0 \Lambda_1^2),\\
w_2^n(t) &\leq \int_0^t  \frac{A}{2} ( x_1(s)^2 +2 \alpha_0 \Lambda_1^2) ds + w_2^n(0) \leq \frac{A}{2} \int_0^t   ( x_1(s)^2 + 2 \alpha_0 \Lambda_1^2) ds +x_2(0) =x_2(t).
\end{align*}
Assume that the assertion is  true for $i = 1, 2, . . . , k$, i.e., if $w_i^n(0) \leq x_i(0)$, then $w_i^n
(t) \leq  x_i(t)$ for $i = 1, 2, . . . , k,$
\begin{align*}
\frac{d}{dt}w_{k+1}^n(t) \leq \frac{A}{2} \sum_{j=1}^{k} (k+1) w_j^n(t) w_{k+1-j}^n(t) \leq  \frac{d}{dt}x_{k+1}(t).
\end{align*}
On  integrating
$$ w_{k+1}^n(t) \leq x_{k+1}(t),$$
which demonstrates that the assumed hypothesis is also valid for $ i = k+1$. As a consequence of the
principle of  mathematical induction, the proof of Lemma \ref{WLEQX} is complete.
\end{proof}

\begin{remark}\label{Remxbd}
It is worth to mention that by replacing the inequality in the proof of Lemma \ref{WLEQX} by equality for $x_i$ in place of $w_i^n$, it can easily be shown that the solution of \eqref{XEQ}--\eqref{XEQIC}, $x=(x_i(t))_{i\in\mathbb{N}}$, for each $i \in \mathbb{N}$ is bounded on $[0,T]$.
\end{remark}
It follows from Lemma \ref{WLEQX} and Remark \ref{Remxbd} that for each $i \in \mathbb{N}$, $w_i^n(t)$ is uniformly bounded on $[0,T]$ i.e.
\begin{align}\label{WUB}
|w_i^n(t)| \leq W_i(T).
\end{align}
We will now show the time equicontinuity  for applying the Arzela–Ascoli theorem \cite{IIV 1995} .

\subsection{\textbf{Time Equicontinuity}} 
To show that for each  $i\in \mathbb{N}$, $w_i^n(t)$ is equicontinuous on $[0,T]$, we exhibit that for each $\epsilon >0$, $\exists$ a $\delta(\epsilon)$ such that
\begin{align*}
(t -\tau) < \delta(\epsilon) \hspace{.7cm} \text{implies} \hspace{.7cm} |w_i^n(t) - w_i^n(\tau)| < \epsilon,
\end{align*}
where $0\leq \tau <t \leq T$.
It follows from equation \eqref{DCBECTIL} that
\begin{align}
|w_i^n(t) -w_i^n(\tau)| =& \Big | \int_{\tau}^t \Big[\frac{1}{2}\sum_{j=1}^{i-1} p_{j,i-j} a_{j,i-j} w_j^n(s) w_{i-j}^n(s) -\sum_{j=1}^{\infty} a_{i,j} w_i^n(s) w_j^n(s) \nonumber \\
&+ \frac{1}{2} \sum_{j=i+1}^{\infty} \sum_{k=1}^{j-1} N_{j-k,k}^i(1- p_{j-k,k}) a_{j-k,k} w_{j-k}^n(s) w_k^n(s) \Big] ds\Big| \nonumber \\
& \leq  \frac{1}{2} \int_{\tau}^t \frac{1}{2}\sum_{j=1}^{i-1} p_{j,i-j} a_{j,i-j} w_j^n(s) w_{i-j}^n(s) ds + \int_{\tau}^t \sum_{j=1}^{\infty} a_{i,j} w_i^n(s) w_j^n(s)ds \nonumber \\
&+ \frac{1}{2} \int_{\tau}^t \sum_{j=i+1}^{\infty} \sum_{k=1}^{j-1} N_{j-k,k}^i (1- p_{j-k,k}) a_{j-k,k} w_{j-k}^n(s) w_k^n(s)  ds \nonumber \\
&:= \vartheta_1^i+ \vartheta_2^i + \vartheta_3^i .\label{H}
\end{align}

 With the help of \eqref{ASSUM} and \eqref{WUB}, $\vartheta_1^i$ can be evaluated  as
\begin{align}
\vartheta_1^i =&\frac{1}{2}\int_{\tau}^t \sum_{j=1}^{i-1} p_{j,i-j} a_{j,i-j} w_j^n(s) w_{i-j}^n(s)ds\nonumber\\ 
&\leq \frac{A}{2}\int_{\tau}^t \sum_{j=1}^{i-1} i w_j^n(s) w_{i-j}^n(s)ds  \nonumber \\
&\leq \frac{A}{2}  \int_{\tau}^t \sum_{j=1}^{i-1}i W_j(T) W_{i-j}(T) ds \nonumber \\
& =\frac{A}{2} V_i(T) (t-\tau), \label{H1}
\end{align}
where 
\begin{align*}
V_i(T) := \sum_{j=1}^{i-1}i W_j(T) W_{i-j}(T).
\end{align*}

Similarly, we  obtain
\begin{align}
\vartheta_2^i &= \int_{\tau}^t \sum_{j=1}^{\infty} a_{i,j} w_i^n(s) w_j^n(s)ds
\leq A W_i(T) (i+1) \Lambda_1 (t-\tau). \label{H2}
\end{align}
Next, we deduce from \eqref{ASSUM}, \eqref{FNP} and \eqref{FM} that 
\begin{align}
\vartheta_3^i &= \frac{1}{2} \int_{\tau}^t \sum_{j=i+1}^{\infty} \sum_{k=1}^{j-1} N_{j-k,k}^i(1- p_{j-k,k}) a_{j-k,k} w_{j-k}^n(s) w_k^n(s)  ds \nonumber\\
&\leq  A \alpha_0  \int_{\tau}^t \sum_{j=2}^{\infty} \sum_{k=1}^{j-1} j w_{j-k}^n(s) w_k^n(s)  ds, \nonumber\\
& \leq 2A \alpha_0 \Lambda_1^2 (t-\tau).\label{H3}
\end{align}
Finally, substituting  the estimates  \eqref{H1}, \eqref{H2} and  \eqref{H3} into  \eqref{H} yields
\begin{align*}
|w_i^n(t) -w_i^n(\tau)| & \leq \Gamma_i (t-\tau),
\end{align*}
 where $$ \Gamma_i :=A \Big(\frac{V_i(T)}{2} + W_i(T) (i+1) \Lambda_1 + 2 \alpha_0 \Lambda_1^2\Big). $$
By choosing  $\delta \leq \frac{\epsilon}{\Gamma_i}$, we obtain the equicontinuity of $w_i^n(t)$ for all $n\geq 2$ and for each fixed $i\in \mathbb{N}$ on  $[0,T].$

Now invoking the Arzela–Ascoli theorem, which warrants  that for each fixed $i\in \mathbb{N}$, the sequence $\{w_i^n\}_{n=2}^{\infty}$ is relatively compact on $\mathcal{C}([0,T], \mathbb{R})$. For $i=1$, we choose a subsequence $w_1^{n_m} \to w_1$ as $m \to \infty$; then for $i=2$, we choose a subsequence $w_2^{n_k} \to w_2$ as $k \to \infty$; and so on. Finally  applying the diagonal process for each $ i\in \mathbb{N}$, we get  a subsequence $ \{w_i^{l}\}_{l=2}^{\infty}$  and a function $ w_i  \in \mathcal{C}([0,T], \mathbb{R})$, such that 
 \begin{align}
  \lim_{l \to \infty}w_i^{l}(t) = w_i(t),  \label{SUBSEQCONVERG}
  \end{align}  
uniformly on $ [0,T]$.
We  now aim to show that the sequence $w(t) = \{w_i(t) \}_{i\in \mathbb{N}}$ is actually a solution of \eqref{DCBEC}--\eqref{IC}.

\subsection{\textbf{Proof of main Theorem}}
Since  the  convergence of subsequence \eqref{SUBSEQCONVERG} stipulates  that for an arbitrary $\epsilon > 0$ there exists $l\in \mathbb{N}$ such that
\begin{align*}
\sum_{i=1}^{\infty} i^{\beta} |w_i(t)- w_i^{l}(t) | < \epsilon.
\end{align*}
As $\epsilon$  and $l$ are arbitrary and owning to  \eqref{betammnt},  we can easily  obtain
\begin{align}\label{WOMEGABETA}
\sum_{i=1}^{\infty} i^{\beta}  w_i \leq \Lambda_{\beta}(T).
\end{align}
Since $1< \beta \leq 2$, hence \eqref{WOMEGABETA} shows that $w =(w_i)_{i\in \mathbb{N}}\in \Omega_{\beta}^+(T)$. In order to prove that $w$ is a solution  of \eqref{DCBEC}--\eqref{IC}, let us consider the following equation with the help of \eqref{DCBEC} and \eqref{DCBECTIL} as 

\begin{align}
\big(w_i^{l}-w_i \big) + w_i =& w_i^0 + \int_0^t \Bigg[ \frac{1}{2} \sum_{j=1}^{i-1}a_{j,i-j} p_{j,i-j} \big(w_{i-j}^{l}(s)-w_{i -j}(s)\big) w_{j}^{l}(s)\nonumber\\
  &+ \frac{1}{2}\sum_{j=1}^{i-1}a_{j,i-j} p_{j,i-j} \big(w_{j}^{l}(s)-w_{j}(s)\big) w_{i-j}(s) \nonumber\\
& -\sum_{j=1}^{\infty} a_{i,j}\big( w_i^{l}(s) - w_i(s) \big) w_j^{l}(s)- \sum_{j=1}^{\infty} a_{i,j}\big(w_j^{l}(s) - w_j(s) \big) w_i(s) \nonumber\\
&+ \sum_{j=i+1}^{\infty} \sum_{k=1}^{j-1} N_{j-k,k}^i(1- p_{j-k,k}) a_{j-k,k}\big( w_{j-k}^{l}(s)-w_{j-k}(s)\big)  w_k^{l}(s) \nonumber\\
& + \sum_{j=i+1}^{\infty} \sum_{k=1}^{j-1} N_{j-k,k}^i(1- p_{j-k,k}) a_{j-k,k}\big( w_k^{l}(s)-w_{k}(s)\big) w_{j-k}(s)  \nonumber\\
&+ \frac{1}{2}\sum_{j=1}^{i-1} p_{j,i-j} a_{j,i-j} w_j(s) w_{i-j}(s)-\sum_{j=1}^{\infty} a_{i,j} w_i(s) w_j(s) \nonumber \\
& + \frac{1}{2} \sum_{j=i+1}^{\infty} \sum_{k=1}^{j-1} N_{j-k,k}^i(1- p_{j-k,k}) a_{j-k,k} w_{j-k}(s) w_k(s) \Bigg] ds. \label{CONVEQN}
\end{align}

As a consequence of   \eqref{FM}, \eqref{betammnt} and \eqref{SUBSEQCONVERG}, it is now easy to pass the limit as $l\to \infty$ in the terms ranging  over the finite sums. On the other hand, passing to the limit in an infinite series involves an additional justification. To overcome this difficulty, we must show that for a given $\epsilon >0 $, there exists a $\mathcal{J}(\epsilon)$ (large enough) such that the tails of infinite series goes to zero as $\epsilon \to 0$. 

  Let $\epsilon >0 $ be given. Then, we consider the tail of the fourth term on the right-hand side of  \eqref{CONVEQN},
\begin{align}
\Bigg|\sum_{j=\mathcal{J}(\epsilon)}^{\infty} a_{i,j} w_j^{l} \Bigg |&\leq A \sum_{j=\mathcal{J}(\epsilon)}^{\infty} (i+j) w_j^{l} \leq  \frac{Ai}{\mathcal{J}(\epsilon)}\sum_{j=\mathcal{J}(\epsilon)}^{\infty} j w_j^{l} +\frac{A}{\mathcal{J}(\epsilon)^{\beta-1}} \sum_{j=\mathcal{J}(\epsilon)}^{\infty} j^{\beta} w_j^l \nonumber \\
&\leq \frac{Ai \Lambda_1}{\mathcal{J}(\epsilon)} + \frac{A \Lambda_{\beta}(T)}{\mathcal{J}(\epsilon)^{\beta-1}}.
\end{align}
Similarly, the tail of the fifth term can be controlled as  
\begin{align}
\Bigg| \sum_{j=\mathcal{J}(\epsilon)}^{\infty} a_{i,j} \big(w_j^{l} - w_j\big) \Bigg|\leq \frac{2A \Lambda_1 i}{\mathcal{J}(\epsilon)} + \frac{2A  \Lambda_{\beta}(T)}{\mathcal{J}(\epsilon)^{\beta-1}}.
\end{align}

Moving on to  the sixth term on the right-hand side of \eqref{CONVEQN},
\begin{align*}
\Bigg| \sum_{j=i+1}^{\infty} \sum_{k=1}^{j-1}& \alpha_0(1- p_{j-k,k}) a_{j-k,k}\big( w_{j-k}^{l}(s)-w_{j-k}(s)\big)  w_k^{l}(s)\Bigg|\nonumber \\
 & \leq \alpha_0 \sum_{j=1}^{\infty}\sum_{k=1}^{\infty}  a_{j,k}\big| w_{j}^{l}(s)-w_{j}(s)\big|  w_k^{l}(s). \nonumber
\end{align*}
Let us now control the tail of the term on the right hand side in the above inequality as
\begin{align*}
\alpha_0 \sum_{k=1}^{\infty}\sum_{j=j=\mathcal{J}(\epsilon)}^{\infty}  a_{j,k}\big| w_{j}^{l}(s)-w_{j}(s)\big|  w_k^{l}(s)&
\leq A\alpha_0 \sum_{k=1}^{\infty}\sum_{j=j=\mathcal{J}(\epsilon)}^{\infty}(j+k)\big| w_{j}^{l}(s)-w_{j}(s)\big|  w_k^{l}(s)\\
&\leq  2A \alpha_0  \Lambda_1\Big(\frac{\Lambda_1}{\mathcal{J}(\epsilon)}+ \frac{\Lambda_{\beta}(T)}{\mathcal{J}(\epsilon)^{\beta-1}}\Big).
\end{align*}

We will select $\mathcal{J}(\epsilon)$ so that $ \max \Big(\frac{\Lambda_1 }{\mathcal{J}(\epsilon)},\frac{ \Lambda_{\beta}(T)}{\mathcal{J}(\epsilon)^{\beta-1}} \Big) \leq \epsilon $. Next  by taking first $ \epsilon \to 0$ and then  $l \to \infty $  on equation \eqref{CONVEQN}, we accomplish,

\begin{align}\label{IVDCBEC}
w_i(t) &= w_i^0 + \int_0^t \Bigg[\frac{1}{2}\sum_{j=1}^{i-1} p_{j,i-j} a_{j,i-j} w_j(s) w_{i-j}(s) -\sum_{j=1}^{\infty} a_{i,j} w_i(s) w_j(s) \nonumber\\
&+ \frac{1}{2} \sum_{j=i+1}^{\infty} \sum_{k=1}^{j-1} N_{j-k,k}^i(1- p_{j-k,k}) a_{j-k,k} w_{j-k}(s) w_k(s) \Bigg] ds.
\end{align}

Based on the above estimates and continuity of $w_i$, we can ensure that the right hand side of \eqref{DCBEC} is also a continuous function of $t\in [0,T].$ Since for each $i \geq 1$, the first term under the integral sign is a continuous function of $t$, whereas second and third term requires additional justification as they contains infinite series.

Thus to prove the continuity of integrand, we must show that the function 
\begin{align}
g(t) =\sum_{j=1}^{\infty} a_{i,j} w_j(t)\hspace{.2cm} \text{and} \hspace{.2cm} h(t) = \frac{1}{2} \sum_{j=i+1}^{\infty} \sum_{k=1}^{j-1} N_{j-k,k}^i(1- p_{j-k,k}) a_{j-k,k} w_{j-k}(t) w_k(t) \label{CF}
\end{align}
are continuous in $t$.
First, let 
\begin{align*}
g(t) = \sum_{i=1}^M a_{i,j} w_j(t) + \sum_{j=M+1}^{\infty} \beta_{i,j} w_j(t)
\end{align*}
Let $t_1, t_2\in [0,T]$ such that
\begin{align*}
\big| g(t_1) - g(t_2) \big| \leq \sum_{j=1}^M \beta_{i,j} |w_j(t_1)- w_j(t_2)| + \sum_{j=M+1}^{\infty} a_{i,j} w_j(t_1)+ \sum_{j=M+1}^{\infty} a_{i,j} w_j(t_2).
\end{align*}
Now invoking the continuity of $w_i$, for all $\epsilon>0$, there exists $N(\epsilon)$  and $\delta(\epsilon,N(\epsilon))$ such that if $|t_1-t_2| < \delta$, then we have 
\begin{align*}
 \sum_{j=1}^M a_{i,j} |w_j(t_1)- w_j(t_2)| < \frac{\epsilon}{3},
\end{align*}
and using \eqref{betammnt}, we conclude that
\begin{align*}
\sum_{j=M+1}^{\infty} a_{i,j} w_j(t_1) < \frac{\epsilon}{3}, \hspace{.4cm}\sum_{j=M+1}^{\infty} a_{i,j} w_j(t_2)< \frac{\epsilon}{3}.
\end{align*}
which shows that $ g$ is continuous. Similarly, we can show that $h$ is a continuous function of $t$.

Hence the integrand in \eqref{IVDCBEC} is continuous. Therefore differentiating equation \eqref{IVDCBEC} with the help of Leibniz's rule guarantee  that $w_i$ is a continuously differentiable solution to \eqref{DCBEC}--\eqref{IC} and finally, from \eqref{WOMEGABETA} we obtain that  $w= \{w_i\}_{i\in \mathbb{N}} \in \Omega_{\beta}^+(T)$.

In the next section, we will examine the uniqueness of classical solution to \eqref{DCBEC}--\eqref{IC}.

\section{\textbf{Uniqueness of classical  solution}}

 Let us assume that the collision kernel  satisfy \eqref{ASSUM1} and $w^0 \in \Omega_{\beta}^+(0)$. Let $w(t)=(w_i(t))_{i\in \mathbb{N}}$ and $v(t)=(v_i(t))_{i\in \mathbb{N}}$ in $\Omega_{\beta}^+(T)$ be two solutions to \eqref{DCBEC}--\eqref{IC} on $[0,T]$, where $T>0$ with the same initial condition $w^0=(w_i^0)_{i\in \mathbb{N}}$. Let $ u := w- v$.\\
 Define 
\begin{align} \label{rho}
 \rho(t) = \sum_{i=1}^{\infty} i |u_i(t)|, 
\end{align}
where 

\begin{align}
u_i(t)  =& w_i(t) -v_i(t) =\int_{0}^{t} \frac{1}{2} \sum_{j=1}^{i-1} p_{j,i-j} a_{j,i-j}[w_j(s) w_{i-j}(s) -v_j(s) v_{i-j}(s)]ds\nonumber \\
 &- \int_{0}^{t}\sum_{j=1}^{\infty} a_{i,j} [ w_i(s) w_j(s) - v_i(s) v_j(s)] ds \nonumber\\
& + \int_{0}^{t} \sum_{j=i+1}^{\infty} \sum_{k=1}^{j-1} (1-p_{j-k,k}) N_{j-k,k}^i a_{j-k,k}[ w_{j-k}(s)w_k(s) - v_{j-k}(s) v_k(s) ] ds.\label{c-d}
\end{align}

Substituting  equation \eqref{c-d} into \eqref{rho}, we get 
\begin{align}
\rho(t) =& \frac{1}{2}\int_{0}^{t}   \sum_{i=1}^{\infty} \sum_{j=1}^{i-1}i \sgn(u_i(s))  p_{j,i-j} a_{j,i-j}[w_j(s) w_{i-j}(s) -v_j(s) v_{i-j}(s)]ds \nonumber \\
&-\int_{0}^{t}\sum_{i=1}^{\infty} \sum_{j=1}^{\infty}i\sgn(u_i(s)) a_{i,j} [ w_i(s) w_j(s) - v_i(s) v_j(s)] ds  \nonumber\\
+&\frac{1}{2}\int_{0}^{t}  \sum_{i=1}^{\infty} \sum_{j=i+1}^{\infty} \sum_{k=1}^{j-1} 
   i \sgn(u_i(s)) (1-p_{j-k,k}) N_{j-k,k}^i a_{j-k,k}  \nonumber\\
   &\hspace{3.5cm}\times  [w_{j-k}(s)w_k(s) - v_{j-k}(s) v_k(s) ]ds. \nonumber
\end{align}

By repeated application of Fubini's theorem in the first and third terms on the right hand side of the preceding equation, and rearranging the indices in summation, we arrive
\begin{align}
\rho(t)  =& \frac{1}{2}\int_{0}^{t}   \sum_{i=1}^{\infty} \sum_{j=1}^{\infty}(i+j) \sgn(u_{i+j}(s)) p_{i,j} a_{i,j} [w_i(s)w_j(s) -v_i(s)v_j(s) ] ds\nonumber\\
&-\int_{0}^{t}\sum_{i=1}^{\infty} \sum_{j=1}^{\infty}i \sgn(u_i(s)) a_{i,j} [ w_i(s) w_j(s) - v_i(s) v_j(s)] ds \nonumber\\ 
&+ \frac{1}{2}\int_{0}^{t}  \sum_{k=1}^{\infty} \sum_{j=1}^{\infty} \Big(\sum_{i=1}^{j+k-1} i\sgn(u_i(s))N_{j,k}^i\Big) (1-p_{j,k}) a_{j,k}[ w_{j}(s)w_k(s) - v_{j}(s) v_k(s) ] ds. \label{rho u-v}
\end{align}
Note that 
$$w_i(s)w_j(s) -v_i(s)v_j(s)   = u_i(s)w_j(s) +v_i(s)u_j(s). $$
With the help of above identity and after rearranging the terms, \eqref{rho u-v} becomes,
\begin{align}
\rho(t)  =& \frac{1}{2}\int_{0}^{t}   \sum_{i=1}^{\infty} \sum_{j=1}^{\infty}[(i+j) \sgn(u_{i+j}(s)) - i \sgn (u_i(s)) - j \sgn(u_j(s))] p_{i,j} a_{i,j}  \nonumber\\
& \hspace{6cm}\times [u_i(s)w_j(s) +v_i(s)u_j(s) ] ds \nonumber\\
&+ \frac{1}{2}\int_{0}^{t}  \sum_{k=1}^{\infty} \sum_{j=1}^{\infty} \Big(\sum_{i=1}^{j+k-1} i\sgn(u_i(s))N_{j,k}^i - j\sgn(u_j(s) - k\sgn(u_k(s))\Big) (1-p_{j,k}) a_{j,k}\nonumber\\
& \hspace{6cm}\times [ w_{j}(s)u_k(s) + v_{k}(s) u_j(s) ] ds. \nonumber \\
\end{align}
This can be rewritten as 
\begin{align}
\rho(t) =& \frac{1}{2}\int_{0}^{t}   \sum_{i=1}^{\infty} \sum_{j=1}^{\infty}\mathcal{P}(i,j,s)  p_{i,j} a_{i,j} u_i(s)w_j(s)  ds +  \frac{1}{2}\int_{0}^{t}   \sum_{i=1}^{\infty} \sum_{j=1}^{\infty}\mathcal{P}(i,j,s)  p_{i,j} a_{i,j}v_i(s)u_j(s)\nonumber \\ 
&+ \frac{1}{2}\int_{0}^{t}  \sum_{k=1}^{\infty} \sum_{j=1}^{\infty}\mathcal{Q}(j,k,s) (1-p_{j,k}) a_{j,k} w_{j}(s)u_k(s)   ds \nonumber \\ 
&+\frac{1}{2}\int_{0}^{t}  \sum_{k=1}^{\infty} \sum_{j=1}^{\infty}\mathcal{Q}(j,k,s) (1-p_{j,k}) a_{j,k}  v_{k}(s) u_j(s):= \sum_{i=1}^4 \mathcal{R}_i(t), \label{R1234} 
\end{align}

where  
\begin{align*}
 \mathcal{P} (i,j,t) :=  (i+j) \sgn(u_{i+j}(s)) - i \sgn (u_i(s)) - j \sgn(u_j(s)), 
 \end{align*}
and
\begin{align*}
\mathcal{Q}(i,j,t) := \sum_{k=1}^{i+j-1} k \sgn(u_k(s))N_{i,j}^k - i \sgn(u_i(s) - j\sgn(u_j(s)).
\end{align*}

Using the properties of signum function, we can evaluate
\begin{align}
\mathcal{P}(i,j,t) u_i(t)& = [(i+j) \sgn(u_{i+j}(t)) - i \sgn (u_i(t)) - j \sgn(u_j(t))]u_i(t) \nonumber\\
& \leq [(i+j) -i +j] |u_i(t)|= 2j |u_i(t)|.\nonumber
\end{align}
Similar to preceding argument, we obtain
\begin{align*}
&\mathcal{P}(i,j,t) u_j(t)   \leq 2i |u_j(t)|,\hspace{.5cm} \mathcal{Q(}i,j,t) u_j(t) \leq 2i |u_j(t)| \\
\hspace{.3cm} &\text{and} \hspace{.3cm} \mathcal{Q}(i,j,t) u_i(t)  \leq 2j|u_i(t)|.
\end{align*}

 Let us evaluate the first term in \eqref{R1234} as
 \begin{align}
 \mathcal{R}_1(t)& = \frac{1}{2}\int_{0}^{t}   \sum_{i=1}^{\infty} \sum_{j=1}^{\infty}\mathcal{P}(i,j,s)  p_{i,j} a_{i,j} u_i(s)w_j(s)  ds \nonumber \\ 
 &  \leq \frac{B}{2}\int_{0}^{t}  \sum_{i=1}^{\infty} \sum_{j=1}^{\infty} 2j |u_i(s)| (i^{\gamma}+j^{\gamma}) w_j(s) ds \nonumber \\
  &\leq B \sup_{s\in[0,t]}(\mathcal{M}_{1}(s) + M_{1+\gamma}(s)) \int_{0}^{t}   \sum_{i=1}^{\infty} i |u_i(s)|  ds \nonumber \\
 &  \leq B \sup_{s\in[0,t]}(\mathcal{M}_{1}(s)+ \mathcal{M}_{\beta}(s)) \int_{0}^{t}   \rho(s)  ds.\nonumber
 \end{align}
 
Analogously, $\mathcal{R}_2(t), \mathcal{R}_3(t)$ and $ \mathcal{R}_4(t)$ can be estimated as
 \begin{align}
 \mathcal{R}_2(t) & \leq  B \sup_{s\in[0,t]}(\mathcal{M}_{1}(s)+ \mathcal{M}_{\beta}(s))  \int_{0}^{t}   \rho(s)  ds, \nonumber
 \end{align}
 
\begin{align}
\mathcal{R}_3(t) & \leq   B \sup_{s\in[0,t]}(\mathcal{M}_{1}(s)+ \mathcal{M}_{\beta}(s))  \int_{0}^{t}   \rho(s)  ds, \nonumber
 \end{align} 
 

 \begin{align}
\mathcal{R}_4(t)&\leq  B \sup_{s\in[0,t]}(\mathcal{M}_{1}(s)+ \mathcal{M}_{\beta}(s))    \int_{0}^{t}   \rho(s)  ds.\nonumber
 \end{align}

 Now gathering the estimates on $\mathcal{R}_1, \mathcal{R}_2, \mathcal{R}_3$, and $\mathcal{R}_4$ and inserting into \eqref{R1234} to obtain
 \begin{align}
 \rho(t) &\leq 4  B \sup_{s\in[0,t]}(\mathcal{M}_{1}(s)+ \mathcal{M}_{\beta}(s))  \int_{0}^{t}   \rho(s)  ds \nonumber \\
 & \leq  \Theta \int_{0}^{t}   \rho(s)  ds, \nonumber
 \end{align}
 where $\Theta=4 B (\Lambda_1 +\Lambda_{\beta}(T))$. Next, the application of Gronwall's inequality gives
 
 $$ \rho(t) \leq 0 \times \exp(\Theta T)=0,$$
 
which implies $w_i(t) = v_i(t) $ for $t \in [0,T]$.

Finally, we will show  the mass conservation property of the solution $w \in \Omega_{\beta}^+(T)$ to \eqref{DCBEC}--\eqref{IC}. 
\section{\textbf{Mass Conservation Property of the solution}}
In order to do so, it is enough to establish that the mass at any time $t$ is same as the  mass taken initially, i.e.,
 $$\mathcal{M}_1(t) = \mathcal{M}_1(0)=\Lambda_1.$$
 Multiplying equation \eqref{DCBEC} from $i$ and taking the summation from $1$ to $ \infty$, we have
 \begin{align}
  \frac{d}{dt}\mathcal{M}_1(t)=\sum_{i=1}^{\infty}i \Big( \frac{1}{2}\sum_{j=1}^{i-1} p_{j,i-j} a_{j,i-j} w_j(t) w_{i-j}(t)\Big) - \sum_{i=1}^{\infty} i \Big(\sum_{j=1}^{\infty} a_{i,j} w_i(t) w_j(t)\Big)  \nonumber \\
  +  \sum_{i=1}^{\infty} i \Big(\frac{1}{2}\sum_{j=i+1}^{\infty} \sum_{k=1}^{j-1} (1-p_{j-k,k}) N_{j-k,k}^i a_{j-k,k} w_{j-k}(t)w_k(t)\Big) \label{MCSUM}
 \end{align}
 Successive application of the results \eqref{ASSUM}, \eqref{LMC}, and \eqref{betammnt} leads to
 \begin{align*}
 \sum_{i=1}^{\infty}i \Big( \frac{1}{2}\sum_{j=1}^{i-1} p_{j,i-j} a_{j,i-j} w_j(t) w_{i-j}(t)\Big) &= \sum_{i=1}^{\infty} \sum_{j=1}^{\infty} i p_{i,j} a_{i,j} w_i(t) w_j(t)\\
  & \leq A \sum_{i=1}^{\infty} \sum_{j=1}^{\infty} i (i+j) w_i(t) w_j(t)\\
  & = A(M_2(T) + M_1(T)) M_1(T)<+\infty.
 \end{align*} 
 Similarly, we obtain
 \begin{align*}
 \sum_{i=1}^{\infty} i \Big(\sum_{j=1}^{\infty} a_{i,j} w_i(t) w_j(t) &= \sum_{i=1}^{\infty} \sum_{j=1}^{\infty}i a_{i,j} w_i(t) w_j(t)\\
  &\leq A(M_2(T) + M_1(T)) M_1(T)<+\infty ,
 \end{align*}
  and 
  \begin{align*}
  \sum_{i=1}^{\infty} i \Big(\frac{1}{2}\sum_{j=i+1}^{\infty} \sum_{k=1}^{j-1} (1-p_{j-k,k}) N_{j-k,k}^i a_{j-k,k}& w_{j-k}(t)w_k(t)\Big)\\
  & =\frac{1}{2}\sum_{j=1}^{\infty} \sum_{k=1}^{\infty} (j+k) a_{j,k}(1-p_{j,k}) w_j(t) w_k(t)\\
  & \leq A(M_2(T) + M_1(T)) M_1(T)<+\infty.
  \end{align*}
  Since all the summation terms on the right hand side of \eqref{MCSUM} are finite. Therefore by Fubini's theorem, on changing the order of summation, we get
  \begin{align*}
   \frac{d}{dt}\mathcal{M}_1(t) = 0.
  \end{align*}
 On integrating with respect to $t$, we  get the desired result.

\section{\textbf{Propagation of Moments }}

In this section, by mean of the following proposition, we show that for a given $ w^{0}= (w_i^{0})_{i\geq 1}\in \Omega_{\beta}^+(0)$ such that $ \sum_{i=1}^{\infty} i^q w_i^0 < \infty $ for some $q > 1$, the solution $w$ to \eqref{DCBEC} -\eqref{IC}, established in Theorem \ref{MT}, has the same features throughout time evolution i.e., $ \sum_{i=1}^{\infty} i^q w_i(t) <\infty$.

\begin{prop}\label{PROP}
Assume that \eqref{ASSUM} and  \eqref{LMC} are fulfilled, and let $w^{0}= (w_i^{0})_{i\in \mathbb{N}}\in \Omega_{\beta}^+(0) $ such that
\begin{align}
 \sum_{i=1}^{\infty} i^q w_i^{0} < \infty \hspace{.2cm}\text{for some} \hspace{.2cm} q>1.\label{INITIALMM}
 \end{align}
Then for each $T>0$, the solution $w$ to \eqref{DCBEC}--\eqref{IC} on $[0,+\infty)$, constructed in Theorem \ref{MT} satisfies
$$\sup_{t \in[0,T]}\sum_{i=1}^{\infty} i^q w_i(t) <\infty. $$
\end{prop}

\begin{proof}
By \eqref{SUBSEQCONVERG}, we know that, for each $t\in [0,+\infty)$,

\begin{align*}
 \lim_{n\to \infty} w_i^{n} (t) =w_i(t),\hspace{.2cm} i \geq 1.
\end{align*}

 where $w_i^n $ refers to the solution to \eqref{DCBECTIL}-\eqref{ICTIL}  given by Theorem \ref{LET}. Taking $g_i=i^q, 1\leq i \leq n $ in \eqref{GME} and following the calculations similar to \eqref{betammnteq}, we obtain

 \begin{align*}
 \frac{d}{dt} \sum_{i=1}^{n} i^q w_i^n &\leq \frac{A \jmath_q }{2} \sum_{i=1}^{n-1} \sum_{j=1}^{n-i} ( ij^q + i^q j)w_i^n w_j^n \\
 & \leq A \jmath_q \|w^{0}\| \sum_{i=1}^n i^q w_i^{0} ,  
 \end{align*}
 and by the application of  Gronwall's lemma, we have
 $$\sum_{i=1}^{n} i^q w_i^n(t) \leq \exp(A \jmath_q \|w^0\| t) \sum_{i=1}^n i^q w_i^{0} \hspace{1cm}t \geq 0.$$

Now,  thanks to \eqref{SUBSEQCONVERG} and \eqref{INITIALMM}, we may pass to the limit as $ n \to \infty$ in the aforementioned inequality
 
 $$\sum_{i=1}^{\infty} i^q w_i(t) \leq  \exp(A \jmath_q \|w^{0}\| t) \sum_{i=1}^{\infty} i^q w_i^{0} \hspace{1cm}t \geq 0.$$
By taking the supremum over $t$, we complete  the proof of Proposition \ref{PROP}.
\end{proof}


\begin{thebibliography}{}

		
\bibitem{BALL 90} Ball, J. M., \& Carr, J., The discrete coagulation-fragmentation equations: Existence,
uniqueness, and density conservation, \textit{J. Stat. Phys.}, \textbf{61}, 203--234, 1990.


\bibitem{BAN 12} Banasiak, J., Global classical solutions of coagulation–fragmentation equations with unbounded coagulation rates, \textit{Nonlinear Anal. Real World Appl.}, \textbf{13}, 91--105, 2012.

\bibitem{BAN 19}  Banasiak, J., Joel, L. O. \& Shindin, S., Discrete growth–decay–fragmentation equation: wellposedness and long-term dynamics, \textit{J. Evol. Equ.}, \textbf{19}, 771--802, 2019.


\bibitem{BAN 19I}  Banasiak, J., Joel, L. O. \& Shindin, S., The discrete unbounded coagulation-fragmentation equation with growth, decay and sedimentation, \textit{Kinet. Relat. Models}, \textbf{12(5)}, 1069--1092, 2019.




\bibitem{BAN 12I} Banasiak, J. \& Lamb, W., The discrete fragmentation equation: semigroups, compactness and asynchronous exponential growth, \textit{Kinet. Relat. Models}, \textbf{5}, 223--236, 2012.


\bibitem{BLL 2019}  Banasiak, J., Lamb, W. \& Lauren\c{c}ot, P., \textit{Analytic Methods for Coagulation-Fragmentation Models, Vol. 1 \& 2}, CRC Press, Boca Raton, FL, 2019.



\bibitem{PKB 2021} Barik, P. K. \&  Giri, A. K., Existence and uniqueness of weak solutions to the singular coagulation equation with collisional breakage,  \textit{Nonlinear Differ. Equ. Appl.}, \textbf{28(34)}, 6115--6133, 2021. 




\bibitem{PKB 2020} Barik, P. K. \&  Giri, A. K., Global classical solutions to the continuous coagulation equation with collisional breakage,  \textit{Z. Angew. Math. Phys.}, \textbf{71(38)}, 1--23, 2020.



\bibitem{PKB 2020I} Barik, P. K. \&  Giri, A. K., Weak solutions to the continuous coagulation model with collisional breakage,  \textit{Discrete Contin. Dyn. Syst. Ser. A}, \textbf{40(11)}, 
6115--6133, 2020. 



\bibitem{PKB 2021I} Barik, P. K., Giri, A. K. \& Kumar, R., Mass-conserving weak solutions to the coagulation and collisional breakage equation with singular rates, \textit{Kinet. Relat. Models}, \textbf{14(2)}, 389--406, 2021.



\bibitem{BRILL 2015} Brilliantov, N., Krapivsky, P. L., Spahn, F., Bodrova, A., Hayakawa, H., Stadnichuk, V. \& Schmidt, J., Size distribution of particles in Saturn's rings from aggregation and
fragmentation, \textit{Proc. Natl. Acad. Sci. USA}, \textbf{112}, 9536--9541, 2015.
 


\bibitem{CARR 92} Carr, J., Asymptotic behaviour of solutions to the coagulation-fragmentation equations. I. The strong fragmentation case, \textit{Proc. R. Soc. Edinb. A}, \textbf{121(34)}, 231--244, 1992.

\bibitem{CARR 94} Carr, J. \&  Da Costa, F. P., Asymptotic behavior of solutions to the coagulation–
fragmentation equations, II, Weak fragmentation,  \textit{J. Stat. Phys.}, \textbf{77}, 89–-123, 1994.




\bibitem{CHNG 90} Cheng, Z. \& Redner, S., Kinetics of fragmentation, \textit{J. Phys. A Math. Gen.}, \textbf{23}, 1233--1258, 1990.



\bibitem{COSTA 2015} Da Costa, F. P., Mathematical Aspects of Coagulation-Fragmentation Equations. In: Bourguignon JP., Jeltsch R., Pinto A., Viana M. (eds) Mathematics of Energy and Climate Change. CIM Series in Mathematical Sciences, vol 2. Springer, Cham, 2015.


\bibitem{COSTA 95} Da Costa, F. P., Existence and uniqueness of density conserving solutions to the
coagulation–fragmentation equations with strong fragmentation,  \textit{J. Math. Anal. Appl.}, \textbf{192}, 892--914, 1995.


\bibitem{PBD 96} Dubovskii, P. B. \&  Stewart, I. W., Existence, uniqueness and mass conservation for the coagulation-fragmentation equation, \textit{Math. Methods Appl. Sci.}, \textbf{19}, 571--591, 1996.



\bibitem{FAF 61} Filippov, A. F., On the distribution of the sizes of particles which undergo splitting, \textit{Theory Probab. its Appl.}, \textbf{6}, 275--294, 1961.





\bibitem{PBD 77} Galkin, V. A. \& Dubovskii, P. B., Solution of the coagulation equation with unbounded kernels, \textit{Differ. Equ.}, \textbf{13}, 1460--1470, 1977.


 


\bibitem{AKG 2021} Giri, A. K. \& Lauren\c{c}ot, P., Weak solutions to the collision-induced breakage equation with dominating coagulation, \textit{J. Differ. Equ.}, \textbf{288}, 690--729,  2021.





\bibitem{AKG 2021I} Giri, A. K. \& Lauren\c{c}ot, P., Existence and non-existence for collision-induced breakage equation, \textit{SIAM J. Math. Anal.}, \textbf{53(4)}, 4605--4636, 2021.


\bibitem{KAPUR 72} Kapur, P., Self-preserving size spectra of comminuted particles, \textit{Chem. Eng. Sci.}, \textbf{27}, 425-431, 1972.


\bibitem{LAMB 20} Kerr, L., Lamb, W. \& Langer, M.,  Discrete fragmentation systems in weighted $\ell^{1}$ spaces, \textit{J. Evol. Equ.}, \textbf{20}, 1419--1451, 2020.




\bibitem{Kostoglou 2000} Kostoglou, M. \&  Karabelas, A. J., A study of the nonlinear breakage equation: analytical and asymptotic solutions, \textit{J. Phys. A Math. Gen.}, \textbf{33}, 1221--1232, 2000.



\bibitem{Krapivsky 2003}  Krapivsky, P. L. \& Ben-Naim, E., Shattering transitions in collision-induced
fragmentation, \textit{Phys. Rev. E}, \textbf{68(2)}, 021102, 2003.


\bibitem{Laurencot 1999} Lauren\c{c}ot, P., Global solutions to the discrete coagulation equations,  \textit{Mathematika}, \textbf{46}, 433--442, 1999.

\bibitem{Laurencot 2001} Lauren\c{c}ot, P. \&  Wrzosek, D., The discrete coagulation equations with collisional breakage, \textit{J. Stat. Phys.}, \textbf{104}, 193--220, 2001.

\bibitem{Laurencot 2002} Lauren\c{c}ot, P., The Discrete Coagulation equations with multiple fragmentation,  \textit{Proc. Edinb. Math. Soc.}, \textbf{45(1)}, 67--82, 2002.




\bibitem{LG 1976} List, R. \& Gillespie, J. R., Evolution of raindrop spectra with collision-induced breakup, \textit{J. Atmos. Sci.}, \textbf{33}, 2007--2013, 1976.



\bibitem{MCG 87} McGrady, E. D. \& Ziff, R.M.,``Shattering'' transition in fragmentation, \textit{Phys. Rev. Lett.}, \textbf{58}, 892--895, 1987.



\bibitem{SV 1972} Safronov, V., Evolution of the Protoplanetary Cloud and Formation of the Earth and the
Planets, \textit{Israel Program for Scientific Translations}, 1972.


\bibitem{MVS 1916} Smoluchowski, M., Drei Vorträge über Diffusion, Brownsche Molekularbewegung und
Koagulation von Kolloidteilchen, \textit{Physik. Zeitschr.}, \textbf{17}, 557--599, 1916.



\bibitem{SIW 1989} Stewart, I. W., A global existence theorem for the general coagulation-fragmentation equation with unbounded kernels, \textit{Math. Methods Appl. Sci.}, \textbf{11}, 627-–648, 1989.



\bibitem{SRC 1978} Srivastava, R. C., Parameterization of raindrop size distributions, \textit{J. Atmos. Sci.}, \textbf{35}, 108--117, 1978.


\bibitem{IIV 1995} Vrabie, I. I., \textit{Compactness Methods for Nonlinear Evolutions, 2nd edition}, Pitman Monogr. Surveys Pure Appl. Math. \textbf{75}, Longman, 1995.


\bibitem{ZRM 85} Ziff, R. M. \& McGrady, E. D., The kinetics of cluster fragmentation and depolymerisation, \textit{J. Phys. A}, \textbf{18}, 3027--3037, 1985.


\end{thebibliography}
\end{document}